\documentclass[11pt]{article}
\usepackage{subfigure}
\usepackage{authblk}
\usepackage{hyperref}
\usepackage{mathtools}
\usepackage{float}

\usepackage{amssymb,amsthm,amsmath,mathrsfs,fullpage,multirow}
\usepackage{tikz}
\usetikzlibrary{matrix}

\usepackage{xparse} 

\NewDocumentCommand{\tubr}{o m}{%
  \tikz[baseline=(X.base)]{
    \def\bracketgap{1.8pt}
    \node[inner sep=0pt, outer sep=0pt] (X) {$#2$};
    \draw[line width=0.4pt] 
      ([yshift=-\bracketgap]X.south west) ++(0,\IfValueTF{#1}{#1}{2pt}) 
      -- ([yshift=-\bracketgap]X.south west)
      -- ([yshift=-\bracketgap]X.south east)
      -- ++(0,\IfValueTF{#1}{#1}{2pt});
  }%
}

\newcommand{\thistheoremname}{}
\newtheorem*{genericthm*}{\thistheoremname}
\newenvironment{namedthm*}[1]
  {\renewcommand{\thistheoremname}{#1}%
   \begin{genericthm*}}
  {\end{genericthm*}}

\newcommand{\Dp}{{\rm dp}}

\makeatletter
\def\keywords{\xdef\@thefnmark{}\@footnotetext}
\makeatother

\renewcommand\S{\mathcal S}

\newcommand\C{\mathcal C}

\newtheorem{theorem}{Theorem}[section]
\newtheorem{lemma}[theorem]{Lemma}
\newtheorem{proposition}[theorem]{Proposition}
\newtheorem{corollary}[theorem]{Corollary}
\newtheorem{conjecture}[theorem]{Conjecture}
\theoremstyle{definition}
\newtheorem{definition}[theorem]{Definition}
  \newtheorem{question}[theorem]{Question}
    \newtheorem{remark}[theorem]{Remark}
    \newtheorem{example}[theorem]{Example}

\usepackage{enumerate}

\title{Characterizing avoidance in cycles via vincular patterns}
\author{Robert P. Laudone}
\affil{{\small Department of Mathematics, United States Naval Academy, Annapolis, MD, 21402}\\{\small Email: laudone@usna.edu }}

\date{}
\begin{document}

\maketitle

\begin{abstract}
    We show that cyclic permutations avoiding $321$ are precisely those permutations whose image under the fundamental bijection avoid a set of vincular patterns. We do this by using pattern functions and arrow patterns, in combination with the characterization of $321$ avoidance in terms of equality of the upper bound of the Daiconis-Graham inequalities. We then explore some consequences of this result, including upper and lower bound results on the growth rate of $321$ avoiding cycles.
\end{abstract}

\section{Introduction}

Richard Stanley originally posed the question of determining the number of cyclic permutations of length $n$ that avoid a permutation $\sigma \in \S_3$ in 2007 at the Permutation Patterns conference, we denote these permutations by $c_n(\sigma)$. Since then, some progress has been made in the study of pattern avoidance in cyclic permutations. For example, the authors of \cite{AE14, BC19} enumerated all cyclic permutations avoiding a pair of permutations with the exception of the pair $c_n(132,213)$. For this remaining case, bounds were established in \cite{H19}. The number of cyclic permutations avoiding a consecutive 123 or 321 pattern, an instance of vincular avoidance (see Definition \ref{def: vinc}), were enumerated in \cite{ET19}. In \cite[Conjecture 5.2]{BC19}, B\'ona and Cory conjectured $2c_{n-1}(\sigma) \leq c_n(\sigma) \leq 4 c_{n-1}(\sigma)$ for $\sigma \in \S_3$. Recently, in \cite{AGL25}, we proved (and sometimes improved) the lower bounds of these conjectures for $\sigma \not= 123$ by showing cyclic permutations avoiding a given pattern admit a partial groupoid structure.

Despite this progress, the question of enumerating cyclic permutations avoiding one pattern of size three remains open.  One of the main challenges in studying pattern avoidance in cyclic permutations is that avoidance is defined using one-line notation, which does not readily reveal a permutation’s cyclic structure. Several efforts have been made to develop tools that bridge the gap between a permutation’s one-line and cycle notations. For instance, work done in \cite{ABBGJ23} examines avoidance in a cyclic permutation's one-line and cycle form. In \cite{AL24}, we studied how avoidance in one-line notation interacts with avoidance in the image of the so-called \emph{fundamental bijection}. In \cite{BT22}, Berman and Tenner introduced pattern functions and arrow patterns to enumerate shallow permutations--originally defined by an equality of the lower bound of the Diaconis-Graham inequalities \cite{DG77}. 


Berman and Tenner's approach also relies on the fundamental bijection, which they use to characterize shallow permutations as those whose fundamental images avoid certain vincular patterns. It is well known that avoiding $321$ is equivalent to the upper bound of the Diaconis-Graham inequality, i.e. permutations whose depth is equal to their inversion number. In this paper, we use this fact along with pattern functions and arrow avoidance to characterize $321$ avoiding cyclic permutations in terms of those permutations whose fundamental bijection avoids a set of vincular patterns. We show,
\begin{namedthm*}{Theorem \ref{thm: 321-final}}
    For $n \geq 1$,
    \[
    \C_n(321) = \theta^{-1}(S_{n,n}(\tubr{32}\, \tubr{41}, \tubr{14}\,\tubr{23}, \tubr{41}\,\tubr{32}, \tubr{23}\,\tubr{14}, \tubr{23} \, \tubr{1})).
    \]
    Where $\C_n(321)$ is the cyclic permutations of length $n$ avoiding $321$, $\theta$ is the fundamental bijection, and $S_{n,n}$ are the permutations of length $n$ with $\pi_1 = n$.
\end{namedthm*}

An immediate consequence of this result is 

\begin{namedthm*}{Corollary \ref{cor: 321-A}}
    For $n \geq 2$,
    \[
    |\C_n(321)| = |\S_{n-1}(\tubr{32}\, \tubr{41}, \tubr{14}\,\tubr{23}, \tubr{41}\,\tubr{32}, \tubr{23}\,\tubr{14}, \tubr{23} \, \tubr{1},\tubr{1} \, \tubr{32})|.
    \]
\end{namedthm*}

This result was somewhat surprising to us because the vincular patterns are small, exhibit nice symmetries and only six patterns are necessary. In Section \ref{sec: background}, we outline all the necessary notation and definitions. In Section \ref{sec: mainResult}, we prove Theorem \ref{thm: 321-final} and Corollary \ref{cor: 321-A}. In the following sections we explore the implications of this result. For example, in Section \ref{sec: convergence}, we show all the cyclic permutations avoiding $321$ admit a semigroup structure and use this to prove $\lim_{n \to \infty} \sqrt[n]{c_n(321)}$ exists. In Section \ref{sec:bounds}, we provide another proof that $\lim_{n \to \infty} \sqrt[n]{c_n(321)} \geq 3.141$, and up to one additional assumption prove the upper bound of B\'ona and Cory's conjecture \cite{BC19}, $c_n(321) \leq 4 c_{n-1}(321)$ for $n \geq 2$. We pose many open questions along the way and conclude with some future directions. Our ultimate hope is that this reframed problem of enumerating vincular avoidance will be more accessible.



\section{Background and Notation} \label{sec: background}

Let $\S_n$ denote the symmetric group on $[n] = \{1,2,\dots,n\}$. We will focus on two representations of a permutation. The first is \textit{one-line notation}, $\pi = \pi_1\pi_2 \ldots \pi_n$ where $\pi_i = \pi(i)$. The second is \emph{cycle notation}, expressing $\pi$ as a product of disjoint cycles.  In this paper, we will always express permutation in their \emph{standard cycle notation}, in which we write cycles with their largest element first, and order the cycles by that largest element. For example, the permutation $\pi = 964572813$ written in its one-line notation, can also be expressed in standard cycle notation as $\pi = (6,2)(9,3,4,5,7,8,1)$. We say that a permutation is \emph{cyclic} if its cycle form consists of a single cycle. We denote the cyclic permutations of length $n$ by $\C_n$. We denote by $\S_{n,k}$ the permutations in $\S_n$ with $\pi_1 = k$.

We say that a permutation $\pi = \pi_1\cdots\pi_n$, written in one-line notation, \emph{contains} $\sigma = \sigma_1\sigma_2\ldots \sigma_k$ if there are indices $i_1 < i_2 < \cdots < i_k$ with $\pi_{i_r} < \pi_{i_s}$ if and only if $\sigma_r < \sigma_s$. In this case, we also say that $\pi_{i_1}\ldots \pi_{i_k}$ is order isomorphic to $\sigma$ and write ${\rm red}(\pi_{i_1}\ldots \pi_{i_k}) = \sigma$. If $\pi$ does not contain $\sigma$, we say $\pi$ \emph{avoids} $\sigma$. We denote the permutations of length $n$ avoiding a permutation $\sigma$ by $\S_n(\sigma)$ and similarly denote the cyclic permutations of length $n$ avoiding a permutation $\sigma$ by $\C_n(\sigma)$.

The \emph{fundamental bijection} $\theta:\S_n\to\S_n$ (described in \cite[Page 30]{S11}) uses both the one-line and standard cycle notation of permutations. For a permutation $\pi \in \S_n$, we obtain $\theta(\pi)$ by writing $\pi$ in its standard cycle notation and reading a new permutation $\theta(\pi)$ in its one-line notation, i.e., by removing the parentheses from the standard cycle notation of $\pi$. For example, using the example in the previous paragraph, we can see that $\theta(964572813) = 629345781.$ For clarity of exposition, we will often write $\theta(\pi)$ as $\hat\pi$. 

Given $\sigma = \sigma_1 \cdots \sigma_n \in \S_n$ and $\tau = \tau_1 \cdots \tau_m \in \S_m$, their \textit{direct sum} is $\sigma \oplus \tau = \sigma_1 \cdots \sigma_n (\tau_1 + n) \cdots (\tau_m + n)$ and their \text{skew sum} is $\sigma \ominus \tau = (\sigma_1 + m) \cdots (\sigma_n + m) \tau_1\cdots \tau_m$. Given a permutation $\pi = \pi_1\ldots\pi_n \in \S_n$, the \textit{reverse} of $\pi$, denoted $\pi^r$ is $\pi_n \ldots \pi_1$. 

\begin{definition} \label{def: vinc}
    A \textit{vincular} pattern is a permutation $\sigma$ in which certain consecutive entries are marked. A permutation $\pi \in \S_n$ contains a vincular pattern if it contains a sub-permutation in the same relative order as $\sigma$ where entries in the sub-permutation corresponding to marked entries in $\sigma$ must occur consecutively in $\pi$. If the first (or last) entry of $\sigma$ is marked, this means the corresponding entry of the sub-permutation must be the first (or last) entry of $\pi$. In keeping with the notation of \cite{BT22}, we will use underbrackets to denote the marked entries.
\end{definition}

For example $\sigma = \tubr{21} \, \tubr{3}$ is a vincular pattern. $\pi = 3214$ contains $\sigma$ since $\tubr{\pi_1\pi_2}\,\tubr{\pi_3} = 324$ is a sub-permutation in the same relative order as $\sigma$ and the marked entries $32$ occur consecutively in $\pi$ and $4$ is the final entry of $\pi$. However, $\pi = 2413$ does not contain $\sigma$ because the only occurrence of $213$ in $\pi$ is $\pi_1\pi_3\pi_4$, but $\pi_1$ and $\pi_3$ are not consecutive in $\pi$.
    
The \textit{count} of a pattern $\sigma$ in $\pi$ is the number of times that $\sigma$ occurs as a sub-permutation of $\pi$. We denote this count by $[\sigma](\pi)$. These counts and their linear combinations are called \textit{pattern functions} (see \cite{BT22}). Many common properties and statistics of permutations can be characterized in terms of pattern functions. For example, the number of descents in a permutation is precisely $[\tubr{21}](\pi)$ the number of occurrences of the vincular pattern $\tubr{21}$.

The upper bound of the Diaconis Graham inequality will play an important role in the coming sections, so we need to define the three statistics involved in this bound. The {\em depth} of a permutation, is a measure of disarray, defined as ${\rm dp}(\sigma) = \sum_{\sigma_i > i} (\sigma_i - i).$ The {\em length} of a permutation is $\ell_s(\sigma) = \sum_{i=1}^n |\{i < j \; | \; \sigma_i > \sigma_j\}|$. This is the number of inversions, which also equals $[21](\sigma)$. The {\em reflection length} of a permutation is $\ell_T(\sigma) = n - {\rm cyc}(\pi)$, where ${\rm cyc}(\pi)$ is the number of cycles in the cycle decomposition of $\pi$.

In \cite{BT22}, Berman and Tenner then develop the notion of \textit{arrow patterns}, which simultaneously contains some of the information from $\pi$ and some from $\hat{\pi}$. We reproduce their definition here because we will use it often,

\begin{definition}{\cite[Definition 4.1]{BT22}} \label{def: arrow}
    An \textit{arrow pattern} $\alpha$ in $\S_k$ consists of
    \begin{itemize}
        \item sets of integers $A = \{a_1,\dots,a_m\}$, $B = \{b_1,\dots,b_h\}$ and $C = \{c_1,\dots,c_h\}$ where $A \cup B \cup C = [k]$. 
        \item a string $\nu = a_1\ldots a_m$ where some elements can be marked. Similar to a vincular pattern, but it is possible that $A$ does not consist of consecutive integers, i.e. $A$ may  not be $[m]$.
        \item a (possibly empty) collection of $h$-arrows, $\{b_i \to c_i \; | \; i = 1,\dots,h\}$.
    \end{itemize}
\end{definition}

We often represent these arrow patterns by $\underset{\scriptscriptstyle b_1 \to c_1, \cdots,b_h \to c_h}\nu$, especially when $|B| = |C| = 1$ and we have an arrow, in which case arrow patterns look like $\underset{\scriptscriptstyle b_1 \to c_1}\nu$. Given a permutation $\pi \in \S_n$ with $\sigma = \theta^{-1}(\pi)$, we say that a sub-string of the one line notation for $\pi$, $\omega \coloneqq x_{a_1}\ldots x_{a_m} = \pi_{t_1}\ldots\pi_{t_m}$, is an occurrence of $\alpha$ if,
\begin{itemize}
    \item the sub-string $\omega$ is order isomorphic to $\nu$.
    \item if $a_j$ and $a_{j+1}$ are marked in $\nu$, then $x_{a_j}$ and $x_{a_{j+1}}$ are adjacent in $\pi$
    \item there exists $X = \{x_1 < \cdots < x_k\} \subset [n]$ such that $\{x_{a_1},\dots, x_{a_m}\} \subseteq X$ and for every arrow $b_i \to c_i$ we have $\sigma(x_{b_i}) = x_{c_i}$ for $1 \leq i \leq h$.
\end{itemize}

We will often start with a permutation $\pi \in \S_n$ and consider how many times $\hat\pi$ contains an arrow pattern $\alpha$. 
For example, consider the arrow pattern $\underset{\scriptscriptstyle 1 \to 4}{\tubr{23}1}$, and the permutation $\pi = 946573812$ with $\hat\pi = 639245781$. $\hat\pi$ contains multiple occurrences of $\tubr{23}1$, for example $392$ and $241$. In both cases, $39$ occur consecutively in $\hat\pi$, as does $24$, but among these two options, only $241$ is an occurrence of the arrow pattern because $\pi_1 = 9$ which is larger than $2$ and $4$. Whereas, $\pi_2 = 4 < 9$.

\section{Proof of Theorem \ref{thm: 321-final}} \label{sec: mainResult}

In this section, we will characterize cyclic 321 avoiding permutations in terms of vincular pattern avoidance in their image under the fundamental bijection. Our main result is the following,

\begin{theorem} \label{thm: 321-final}
    $\pi \in \C_n(321)$ if and only if $\hat\pi_1 = n$ and $\hat\pi$ avoids $\{\tubr{32}\, \tubr{41}, \tubr{14}\,\tubr{23}, \tubr{41}\,\tubr{32}, \tubr{23}\,\tubr{14}, \tubr{23} \, \tubr{1}\}$.
    We can rephrase this as,
    \[
    \C_n(321) = \theta^{-1}(S_{n,n}(\tubr{32}\, \tubr{41}, \tubr{14}\,\tubr{23}, \tubr{41}\,\tubr{32}, \tubr{23}\,\tubr{14}, \tubr{23} \, \tubr{1})).
    \]
\end{theorem}

Another characterization without requiring fixed terms at the start of the permutation which follows immediately from Theorem \ref{thm: 321-final} is,

\begin{corollary}\label{cor: 321-A}
    For $n \geq 2$,
\[
|\C_n(321)| = |\S_{n-1}(\tubr{32}\, \tubr{41}, \tubr{14}\,\tubr{23}, \tubr{41}\,\tubr{32}, \tubr{23}\,\tubr{14}, \tubr{23} \, \tubr{1},\tubr{1} \, \tubr{32})|.
\]
\end{corollary}

\begin{proof}
    By Theorem \ref{thm: 321-final}, $\pi \in \C_n(321)$ if and only if $\hat\pi_2\cdots \hat\pi_n \in \S_{n-1}$ avoids 
    \[
    \tubr{32}\, \tubr{41}, \tubr{14}\,\tubr{23}, \tubr{41}\,\tubr{32}, \tubr{23}\,\tubr{14}, \tubr{23} \, \tubr{1},\tubr{1} \, \tubr{32}.
    \]
    Indeed, by removing $n$ at the start of $\hat\pi$, the only pattern counts we remove are occurrences of $[\tubr{41} \, \tubr{32}]$ that begin with $n$. These are counted instead by $[\tubr{1} \, \tubr{32}]$. The result follows.
\end{proof}

To prove these results, we begin by considering pattern functions that characterize $321$ avoidance in cycles via equality of the upper bound of the Diaconis-Graham inequalities.

\begin{lemma} \label{lem: 321-cyclic-first}
    $\pi \in \C_n(321)$ if and only if $\hat\pi_1 = n$ and:
    \[
    ([2 \; \tubr{31}] +  [\tubr{14} \, \tubr{23}] +  [\tubr{41} \, \tubr{32}] -  [\tubr{24} \; \tubr{13}] -  [\tubr{31} \; \tubr{42}])(\hat\pi) = 0.
    \]
Furthermore, this pattern count is equivalent to
\begin{equation} \label{eq: 321-count}
([\tubr{231}] + [\tubr{23} \, \tubr{41}] + [\tubr{24} \; \tubr{31}] + [\tubr{32} \, \tubr{41}] + [\tubr{14} \, \tubr{23}] +  [\tubr{41} \, \tubr{32}] -  [\tubr{24} \; \tubr{13}])(\hat\pi) = 0.
\end{equation}
\end{lemma}

\begin{proof}
    It is well known that a permutation avoids $321$ if and only if $\Dp(\pi) = \ell_S(\pi)$. When $\pi$ is cyclic, $\ell_T(\pi) = n-1$ so by \cite[Theorem 4.4]{BT22},
    \[
    \Dp(\pi) = n-1 + ([2 \; \tubr{31}] + [\tubr{14} \, \tubr{23}] + [\tubr{41} \, \tubr{32}] + [\tubr{24} \; \tubr{13}] + [\tubr{31} \; \tubr{42}])(\hat\pi)
    \]
    and by \cite[Theorem 4.5]{BT22},
    \[
    \ell_S(\pi) = n-1 + 2([2 \; \tubr{31}] + [\tubr{14} \, \tubr{23}] + [\tubr{41} \, \tubr{32}])(\hat\pi)
    \]
    so we have equality precisely when
    \[
    ([2 \; \tubr{31}] +  [\tubr{14} \, \tubr{23}] +  [\tubr{41} \, \tubr{32}] -  [\tubr{24} \; \tubr{13}] -  [\tubr{31} \; \tubr{42}])(\hat\pi) = 0.
    \]
    We implicitly used the cyclic assumption, so we do also need to require that $\hat\pi_1 = n$. This establishes the first claim. The second comes from expanding $[2\; \tubr{31}]$. Consider what element comes after the $2$, doing so we get the equality
    \[
    [2 \; \tubr{31}] = [\tubr{32} \, \tubr{41}] + [\tubr{31} \; \tubr{42}] + [\underset{\scriptscriptstyle 2 \to 2}{2 \; \tubr{31}}] + [\underset{\scriptscriptstyle 2 \to 3}{2 \, \tubr{41}}] + [\underset{\scriptscriptstyle 2 \to 4}{2 \; \tubr{31}}] + [\underset{\scriptscriptstyle 2 \to 3}{\tubr{231}}].
    \]
    Where the last four counts are of arrow patterns (Definition \ref{def: arrow}). We can make simplifications here because we are only counting these patterns in cyclic permutations. This means that $[\underset{\scriptscriptstyle 2 \to 2}{2 \; \tubr{31}}]$ will never occur. $[\underset{\scriptscriptstyle 2 \to 3}{2 \, \tubr{41}}] = [\tubr{23} \, \tubr{41}]$ because the $3$ in the image of $2$ could not be the largest element in its cycle since $n$ is, and $41$ must come after $2$. Similarly, $[\underset{\scriptscriptstyle 2 \to 4}{2 \; \tubr{31}}] = [\tubr{24} \; \tubr{31}]$, and $[\underset{\scriptscriptstyle 2 \to 3}{\tubr{231}}]= [\tubr{231}]$. This is also forced because $3$ could not be the start of a new cycle because there is only one cycle. This means we have the following equality of pattern functions for cyclic permutations,
    \[
    [2 \; \tubr{31}] = [\tubr{32} \, \tubr{41}] + [\tubr{31} \; \tubr{42}] + [\tubr{23} \, \tubr{41}] + [\tubr{24} \; \tubr{31}] + [\tubr{231}].
    \]
    
    Substituting this in for $[2 \; \tubr{31}]$ in our first pattern function and simplifying we find,
    \[
    ([\tubr{231}] + [\tubr{23} \, \tubr{41}] + [\tubr{24} \; \tubr{31}] + [\tubr{32} \, \tubr{41}] + [\tubr{14} \, \tubr{23}] +  [\tubr{41} \, \tubr{32}] -  [\tubr{24} \; \tubr{13}])(\hat\pi) = 0.
    \]
\end{proof}

So we are almost able to describe $321$ avoiding cyclic permutations in terms of vincular avoidance. We just need to remove the subtraction in the pattern function. Using arrow patterns, we can characterize it in terms of avoidance in the next result,

\begin{lemma}  \label{lem: 321-cyclic}
    $\pi \in \C_n(321)$ if and only if $\hat\pi_1 = n$ and
    \[
    ([\tubr{32} \, \tubr{41}] + [\tubr{14} \, \tubr{23}] + [\tubr{41} \, \tubr{32}] + [\tubr{23} \, 1]_{1 \to 4})(\hat\pi) = 0.
    \]
\end{lemma}

\begin{proof}
    We begin with Lemma \ref{lem: 321-cyclic-first}, which tells us that $\pi \in \C_n(321)$ if and only if $\hat\pi_1 = n$ and 
    \[
    ([\tubr{231}] + [\tubr{23} \, \tubr{41}] + [\tubr{24} \; \tubr{31}] + [\tubr{32} \, \tubr{41}] + [\tubr{14} \, \tubr{23}] +  [\tubr{41} \, \tubr{32}] -  [\tubr{24} \; \tubr{13}])(\hat\pi) = 0.
    \]
    Our goal is to express this as a sum of pattern counts. To do this, consider the terms that can come before, and after, the $1$ in the pattern count $[\tubr{23} \, 1]$. Doing this we find
    \[
    [\tubr{23} \, 1] = [\tubr{34} \; \tubr{12}] + [\tubr{34} \; \tubr{21}] + [\tubr{24} \; \tubr{13}] + [\underset{\scriptscriptstyle 1 \to 4}{\tubr{23} \, 1}] = [\tubr{34} \; \tubr{12}] + [\tubr{34} \; \tubr{21}] + [\tubr{24} \; \tubr{31}] + [\tubr{23} \, \tubr{41}] + [\tubr{231}].
    \]
    Canceling like terms and solving for $[\tubr{24} \; \tubr{13}]$,
    \[
    [\tubr{24} \; \tubr{13}] = [\tubr{24} \; \tubr{31}] + [\tubr{23} \, \tubr{41}] + [\tubr{231}] - [\underset{\scriptscriptstyle 1 \to 4}{\tubr{23} \, 1}].
    \]
    Substituting this in to the pattern function from Lemma \ref{lem: 321-cyclic-first} and simplifying we find that $\pi \in \C_n(321)$ if and only if $\hat\pi_1 = n$ and
    \[
    ([\tubr{32} \, \tubr{41}] + [\tubr{14} \, \tubr{23}] + [\tubr{41} \, \tubr{32}] + [\tubr{23} \, 1]_{1 \to 4})(\hat\pi) = 0.
    \]
    Since this is now a positive sum, we conclude that a permutation is cyclic and avoids $321$ if and only if $\hat\pi_1 = n$ and $\hat\pi$ avoids $\tubr{32} \, \tubr{41}, \tubr{14} \, \tubr{23}, \tubr{41} \, \tubr{32}$ and $\underset{\scriptscriptstyle 1 \to 4}{\tubr{23} \, 1}$.
\end{proof}

This completely characterizes cyclic $321$ avoiding permutations in terms of vincular and arrow pattern avoidance. In particular, we can ignore the cycle decomposition of the permutation and just consider avoidance. We next characterize this arrow pattern avoidance in cycles in terms of normal avoidance.

\begin{lemma} \label{lem:231-arrow}
    A cyclic permutation $\pi \in \C_n(321)$ has
    \[
    [\underset{\scriptscriptstyle 1 \to 4}{\tubr{23} \, 1}](\hat\pi) = ([\tubr{23}\, \tubr{14}] + [\tubr{23}\, \tubr{1}])(\hat\pi),
    \]
    where recall that $\tubr{1}$ signifies that the $1$ must be the final entry in $\hat\pi$
\end{lemma}

\begin{proof}
    In a cyclic permutation, $\hat\pi_1 = n$. To have an occurrence of $\underset{\scriptscriptstyle 1 \to 4}{\tubr{23} \, 1}$ we must have $\tubr{xy} \; z$ in $\hat\pi$ with $z < x < y < \pi(z)$. Since $\pi$ is cyclic, the only way $\pi(z)$ appears before $z$ in $\hat\pi$ is if it is the first entry and $z$ is the final entry of $\hat\pi$, so $\tubr{xy} \; z$ is a $\tubr{23}\, \tubr{1}$ pattern. The other option is that $\pi(z)$ appears after $z$ in which case $\tubr{xy} \; \tubr{z \pi(z)}$ forms a $\tubr{23}\, \tubr{14}$ pattern in $\hat\pi$.
\end{proof}

The following finalized form of our main theorem follows,

\begin{proof}[Proof of Theorem \ref{thm: 321-final}]
    Substitute the pattern function for $[\underset{\scriptscriptstyle 1 \to 4}{\tubr{23} \, 1}]$ from Lemma \ref{lem:231-arrow} into Lemma \ref{lem: 321-cyclic}, the result follows immediately.
\end{proof}

So understanding cyclic permutations avoiding $321$ is equivalent to understanding permutations who avoid five vincular patterns and begin with $n$. Or equivalently by Corollary \ref{cor: 321-A}, understanding permutations who avoid six vincular patterns.

From Corollary \ref{cor: 321-A}, we get an immediate, but known, corollary by strengthening the avoidance conditions,

\begin{corollary} \label{cor: 132-231}
    For $n \geq 2$, let $c_n = |\C_n(321)|$
    \[
    c_n \geq |\S_{n-1}(132,231)| = 2^{n-2}
    \]
\end{corollary}

\begin{proof}
    \sloppy This follows immediately from Corollary \ref{cor: 321-A} because $\S_{n-1}(132,231) \subset \S_{n-1}(\tubr{32}\, \tubr{41}, \tubr{14}\,\tubr{23}, \tubr{41}\,\tubr{32}, \tubr{23}\,\tubr{14}, \tubr{23} \, \tubr{1},\tubr{1} \, \tubr{32})$.
\end{proof}

This is already known from \cite[Theorem 5.3]{BC19} where they show $c_n \geq 2 c_{n-1}$. But this is a large strengthening of the avoidance conditions. It is possible one could strengthen, or weaken, them in lesser ways to get lower, or upper, bounds on the growth rate. We mention this again in the further works section. Our results rely on the cyclic assumption, since we need $\ell_T(\pi) = n-1$, and we often use that $\hat\pi_1 = n$, but it would be interesting to describe avoidance in $\pi$ in terms of avoidance in $\hat\pi$. We conjecture,



    

\begin{conjecture}
    If $\pi \in S_n(321)$ then $\hat\pi$ avoids $\{\tubr{32} \, \tubr{41}, \tubr{23}\, \tubr{14}, \tubr{41} \, \tubr{32}, \tubr{23}\, \tubr{1}\}$.
\end{conjecture}

 This is not an if and only if, but it would be very interesting to find an exact set of vincular patterns that describes $321$ avoidance. In the coming sections, we explore some consequences of Theorem \ref{thm: 321-final}, mentioning many questions along the way.





\section{Convergence of Growth Rate} \label{sec: convergence}

Theorem \ref{thm: 321-final} allows us to define a semigroup structure on $\bigcup_{n \geq 1} \C_n(321)$. Given $\sigma \in \C_n(321)$ and $\tau \in \C_m(321)$, let $\hat\tau_k = 1$ and define a new binary operation $\odot$ as follows,
\[
\hat\tau \odot \hat\sigma = (\hat\tau_1 + n) \cdots (\hat\tau_k + n) \hat\sigma_1 \hat\sigma_2 \cdots \hat\sigma_n (\hat\tau_{k+1} + n) \cdots (\hat\tau_m + n).
\]

\begin{lemma} \label{lem: monoidStructure}
    Given $\sigma \in \C_n(321)$ and  $\tau \in \C_m(321)$, $\theta^{-1}( \hat\tau \odot \hat\sigma) \in \C_{n+m}(321)$.
\end{lemma}

\begin{proof}
    By Theorem \ref{thm: 321-final} since $(\hat\tau \odot\hat\sigma)_1 = m+n$ by construction, it suffices to prove it also avoids $\tubr{32}\, \tubr{41}, \tubr{14}\,\tubr{23}, \tubr{41}\,\tubr{32}, \tubr{23}\,\tubr{14}, \tubr{23} \, \tubr{1}$. Notice, there will never be an occurrence of any of these patterns consisting solely of entries from $\hat\tau$ or $\hat\sigma$ because both $\tau$ and $\sigma$ are cyclic and avoid $321$.

    Since all the entries from $\hat\sigma$ are smaller than all the entries from $\hat\tau$, the only way $\hat\tau \odot \hat \sigma$ could contain a $\tubr{32}\, \tubr{41}$ pattern, is if $\tubr{41}$ was $(\hat\tau_k + n) \hat\sigma_1$ and $\tubr{32}$ was $(\hat\tau_i + n) (\hat\tau_{i+1} + n)$ for $i+1 < k$. But this is not possible because $\hat\tau_k = 1$. 

    Similarly, a $\tubr{14}\,\tubr{23}$ pattern could only occur if $\tubr{14}$ was $\hat\sigma_n (\hat\tau_{k+1} + n)$ and $\tubr{23}$ was $(\hat\tau_i + n) (\hat\tau_{i+1} + n)$ for $k < i$. But then $\hat\tau_k\hat\tau_{k+1}\hat\tau_{i}\hat\tau_{i+1}$ would be a $\tubr{14}\, \tubr{23}$ pattern in $\hat\tau$ because $\hat\tau_k = 1$. To contain a $\tubr{41}\,\tubr{32}$ pattern, we must have $\tubr{41}$ be $(\hat\tau_k + n) \hat\sigma_1$ and $\tubr{32}$ be $\hat\sigma_i\hat\sigma_{i+1}$. But then $\hat\sigma_1 = n$ so this is impossible. To contain a $\tubr{23} \, \tubr{14}$ pattern, we must have $\tubr{14}$ be $\hat\sigma_n (\hat\tau_{k+1} + n)$ and $\tubr{32}$ be $\hat\sigma_i\hat\sigma_{i+1}$. But then $\hat\sigma_i\hat\sigma_{i+1}\hat\sigma_n$ would be a $\tubr{23} \tubr{1}$ pattern in $\hat\sigma$. Finally, to contain a $\tubr{23} \, \tubr{1}$ pattern, since all the entries of $\hat\sigma$ are smaller than $\hat\tau_m+n$, this could only occur if there was already a $\tubr{23} \, \tubr{1}$ pattern in $\hat\tau$, which is not true.

     By Theorem \ref{thm: 321-final}, $\theta^{-1}(\hat\tau \odot\hat\sigma) \in \C_n(321)$.
\end{proof}

\begin{example}
    For example, consider $\hat\sigma = 7531642$ and $\hat\tau = 86421753$ which has $\sigma \in C_7(321)$ and $\tau \in C_8(321)$.
    \[
    \hat\tau \odot \hat\sigma = 15 \; 13 \; 11 \; 9 \;{\bf  8 \; 7 \; 5 \; 3 \; 1 \; 6 \; 4 \; 2\;} 14 \; 12 \; 10
    \]
    This corresponds to the permutation $6\; 14\; 1\; 2\;3\;4\;5\;7\;8\;15\;9\;10\;11\;12\;13$, which avoids $321$.
\end{example}

\begin{remark}
    This idea of placing a partial groupoid structure on cyclic permutations avoiding a pattern $\sigma \in \S_3$ to construct new cyclic permutations from others is the basis for the paper \cite{AGL25}. The construction $\hat\tau \odot \hat\sigma$ is different that the construction we used for $\C_n(321)$ in that paper. It is easier here to prove that $\theta^{-1}(\hat\tau \odot \hat\sigma)$ is still cyclic and avoids $321$ since Theorem \ref{thm: 321-final} characterizes cyclic avoidance in terms of avoidance in the fundamental bijection. We also never need the associativity of $\odot$, but it is not hard to verify.
\end{remark}

Throughout the rest of the paper, we let $c_n = |\C_n(321)|$, suppressing the $321$ for ease of notation. An immediately consequence of Lemma \ref{lem: monoidStructure} is,

\begin{corollary}
    For $n,m \geq 1$ we have
    \[
    c_n c_m \leq c_{n+m}
    \]
\end{corollary}

\begin{proof}
    This immediately follows by noticing that the map from $\C_n(321) \times \C_m(321) \to \C_{n+m}(321)$ defined by,
    \[
    (\tau, \sigma) \to \hat\tau \odot \hat\sigma
    \]
    is an injection. Indeed, suppose $\hat\tau_1 \odot \hat\sigma_1 = \hat\tau_2 \odot \hat\sigma_2$ we claim that $\tau_1 = \tau_2$ and $\sigma_1 = \sigma_2$. First, notice $\hat\sigma$ in $\hat\tau \odot \hat\sigma$ is precisely the $n$ elements that occur after $n+1$. From this we conclude that $\sigma_1 = \sigma_2$. Similarly, one can recover $\hat\tau$ as the elements in $\hat\tau \odot \hat\sigma$ that are larger than $n$, so $\tau_1 = \tau_2$.
\end{proof}

We can use this to conclude,

\begin{proposition}
    The limit $\lim_{n \to \infty} \sqrt[n]{c_n} = \limsup_{n \to \infty} \sqrt[n]{c_n}$ exists.
\end{proposition}

\begin{proof}
    By Lemma \ref{lem: monoidStructure}, for $n,m \geq 1$, $c_n \cdot c_m(321) \leq c_{n+m}(321)$. By Fekete's lemma, we conclude that the limit $\lim_{n \to \infty} \sqrt[n]{c_n}$ exists.
\end{proof}

    

We believe a stronger result is true,

\begin{conjecture}
    For $n \geq 1$, $\sqrt[n]{c_n}$ is monotonically increasing.
\end{conjecture}

If we knew this, it would obviously imply the previous result because $\sqrt[n]{c_n}$ is bounded.

\section{Structure of $\mathcal{A}_n$} \label{sec: An}

We will now focus on the set of permutations avoiding the six patterns from Corollary \ref{cor: 321-A}. We will let $\mathcal{A}_n = \S_{n}(\tubr{32}\, \tubr{41}, \tubr{14}\,\tubr{23}, \tubr{41}\,\tubr{32}, \tubr{23}\,\tubr{14}, \tubr{23} \, \tubr{1},\tubr{1} \, \tubr{32})$. By Corollary \ref{cor: 321-A}, $\mathcal{A}_n$ can equivalently be characterized as the permutations $\tau \in \S_{n}$ with $\theta^{-1}(1 \ominus \tau) \in \C_{n+1}(321)$. We let $a_n = |\mathcal{A}_n|$. So Corollary \ref{cor: 321-A} states $c_n = a_{n-1}$ for $n \geq 2$.

In this section, we explore the structure of $\mathcal{A}_n$. We first use the description in Corollary \ref{cor: 321-A} to explicitly construct permutations in $\mathcal{A}_n$.

\begin{lemma}
    Given any permutation $\pi = \pi_1 \pi_2 \cdots \pi_{k-1} 1 \pi_{k+1} \cdots \pi_n \in \S_n$, if $\pi_1 > \pi_2 > \cdots > \pi_{k-1} > 1$ and $\pi_{k+1} < \pi_{k+2} < \cdots < \pi_n$ then $\pi \in \mathcal{A}_n$.
\end{lemma}

\begin{proof}
    This is a restatement of Corollary \ref{cor: 132-231}.
\end{proof}

We include this restatement just to get a feel for what the permutations in $\mathcal{A}_n$ can look like. The following result is also a restatement of a known result about $\C_n(321)$ which illustrates how the structure of those permutations is reflected in the $\mathcal{A}_{n-1}$,

\begin{lemma} \label{lem: reflect}
    Given $\pi \in \mathcal{A}_n$, $\pi^r \in \mathcal{A}_n$
\end{lemma}

\begin{proof}
    All of the patterns that characterize $\mathcal{A}_n$ in Corollary \ref{cor: 321-A} are invariant under taking the reverse.
\end{proof}

Reversing a permutation in $\mathcal{A}_n$ corresponds to taking the inverse of the corresponding permutation in $\C_{n+1}(321)$. So this is another way to show that if $\pi \in \C_n(321)$, then so is $\pi^{-1}$. 

\begin{lemma}
    Given $\pi \in \mathcal{A}_n$, we cannot have $\pi_1 < \pi_2$ and $\pi_{n-1} > \pi_n$.
\end{lemma}

\begin{proof}
    If $\pi_1 < \pi_n$, then $\pi_1\pi_{n-1}\pi_{n}$ is a $\tubr{1}\,\tubr{32}$. If $\pi_1 > \pi_n$ then $\pi_1\pi_2\pi_n$ is a $\tubr{23} \,\tubr{1}$.
\end{proof}

This means that if $\pi_1 < \pi_2$ we must have $\pi_{n-1} < \pi_n$, and if $\pi_{n-1} > \pi_n$ we must have $\pi_1 > \pi_2$. We could also have $\pi_1 > \pi_2$ and $\pi_{n-1} < \pi_n$. 

\subsection{Simple Permutations}
In this subsection, we will construct all permutations in $\mathcal{A}_n$ by inflating simple permutations in a very specific way. Recall that a permutation $\pi \in \S_n$ is \emph{simple} if there is no proper subset $[i,j] \subsetneq [n]$ where ${\rm red}(\pi_{i} \pi_{i+1} \cdots \pi_{j}) = [1,j-i+1]$. For example, $\pi = 231$ is not simple because ${\rm red}(\pi_1\pi_2) = 12$, but $\pi = 2413$ is simple because every consecutive sub-permutation is not equal to a consecutive interval. 

Given $\pi = \pi_1 \ldots \pi_n \in \S_n$ and $\sigma = \sigma_1 \ldots \sigma_k \in S_k$, the {\em inflation} of $\pi$ by $\sigma$ at $a$ where $1 \leq a \leq n$ is the permutation obtained by replacing $a$ in $\pi$ with $(\sigma_1 + a - 1)\cdots(\sigma_k+a-1)$, and shifting all entries in $\pi$ larger than $a$ up by $k-1$. For example, the inflation of $\pi = 4132$ by $\sigma = 321$ at $a = 3$ is $61{\bf543}2$, where the elements corresponding to $\sigma$ are bold.

We will most often be inflating permutations at $1$, so we denote by $\pi[\sigma]$ the inflation of $\pi$ at $1$ by $\sigma$, so that 
\[
\pi[\sigma] = (\pi_1+k-1) \ldots (\pi_{k}+k-1) \sigma_1 \ldots \sigma_k (\pi_{k+2}+k-1) \ldots (\pi_n+k-1)
\]
where $\pi_{k+1} = 1$. For example, $213[1342]= 5{\bf 1342}6$. This looks similar to our $\odot$ operator from Section \ref{sec: convergence}, but it is slightly different since we are replacing $1$ with $\sigma$, rather than just inserting $\sigma$.

\begin{lemma} \label{lem: inflation}
    Given $\pi \in \mathcal{A}_n$ and $\sigma \in \mathcal{A}_m$, $\pi[\sigma] \in \mathcal{A}_{n+m-1}$.
\end{lemma}

\begin{proof}
    \sloppy By Corollary \ref{cor: 321-A}, it suffices to show that $\pi[\sigma]$ avoids $\{[\tubr{32}\, \tubr{41}], [\tubr{14}\, \tubr{23}], [\tubr{41}\, \tubr{32}], [\tubr{23} \,\tubr{14}], [\tubr{23} \, \tubr{1}], [\tubr{1} \, \tubr{32}]\}$. Since we are inflating $1$ by a permutation $\sigma \in \mathcal{A}_m$, the only patterns we could create are $[\tubr{41} \, \tubr{32}]$ and $[\tubr{23}\, \tubr{14}]$. We avoid all the others because $\pi$ avoids them.
    
    But we cannot create these patterns because a $[\tubr{41} \, \tubr{32}]$ would imply that $\sigma$ contains a $\tubr{1} \, \tubr{32}$ pattern and similarly a $[\tubr{23}\, \tubr{14}]$ implies that $\sigma$ contains a $\tubr{23}\, \tubr{1}$ pattern.
\end{proof}

We can describe all permutations in $\mathcal{A}_n$ by inflating $1$ in simple permutations:

\begin{lemma} \label{lem: infateAt1}
    If $\tau \in \mathcal{A}_n$ is not simple, we have $\tau = \pi[\sigma]$ for a unique simple $\pi \in \mathcal{A}_m$, and $\sigma \in \mathcal{A}_{n-m+1}$.
\end{lemma}

\begin{proof}
    Every non-simple permutation is an inflation of a unique simple permutation (see \cite{AA05}). By Lemma \ref{lem: inflation}, we know that $\pi[\sigma] \in \mathcal{A}_{n}$. We will show that inflating any simple permutation $\pi \in \mathcal{A}_m$ by $\sigma \in \mathcal{A}_{n-m+1}$ containing a $\tubr{12}$ pattern, at any element that is not $1$, produces a permutation that is not in $\mathcal{A}_n$.

    Suppose we inflate by $\sigma$ at $a$ for $1 < a \leq n$, call this inflation $\pi'$. We wish to argue that $\pi' \not\in \mathcal{A}_n$. Consider $a-1$ in $\pi$. It must be the case that every element in $\pi$ to the right of $a-1$ up until we possibly reach $a$ is less than or equal to $a$. Indeed, the element directly to the right of $a-1$ must be at most $a$, otherwise we find a $\tubr{14} \, \tubr{23}$ or $\tubr{23}\, \tubr{14}$ pattern in $\pi'$. But now that the element directly to the right of $a-1$ is less than or equal to $a$, we can continue to apply the same logic to conclude that the element directly to its right is at most $a$ as well. This continues until we reach the end of $\pi$, or $a$.

    By identical reasoning, we can conclude that every element of $\pi$ to the right of $a-i$ up until we possibly reach $a$ is less than or equal to $a$ for any $1 \leq i \leq a-1$.

    One immediate consequence of this is that every element less than $a$ appears to the left of $a$ in $\pi$. If not, then the previous paragraphs force $\pi_n < a$. But then $\pi' \not \in \mathcal{A}_n$ since it would contain a $\tubr{23}\, \tubr{1}$ pattern.

    But if every element less than $a$ appears to the left of $a$ in $\pi$, and every element to the right of an element less than $a$, up until $a$, must also be less than $a$, this forces $\pi$ to have the form $ABaC$ where $B = \{1,\dots,a-1\}$ and $A \sqcup C = \{a+1,\dots,n\}$. However, if $1 < a < n$, $B \not= \emptyset$ and so $Ba$ is a nontrivial subinterval in $\pi$. If $a=n$, then we must have $B = \{1,\dots,n-1\}$ so that $C = \emptyset$, meaning $\pi_n = n$. In either case, we contradict the simplicity of $\pi$.

    So we cannot inflate $\pi$ by $\sigma$ containing $\tubr{12}$ at any $a$ with $1 < a \leq n$. Taking reflections, by Lemma \ref{lem: reflect}, and because the reflection of a simple permutation is still simple, we also cannot inflate by $\sigma$ containing $\tubr{21}$ at any $a$ for $1 < a \leq n$. Since every permutation contains one of those two patterns, and every permutation is an inflation of a simple permutation, we must inflate at $1$ and the result follows.
\end{proof}



We tried to enumerate $\mathcal{A}_n$ by first enumerating the simple permutations. Although this seems difficult, understanding the simple permutations could lead to improved lower, or maybe upper, bounds on the growth rate of $c_n$. As evidence for this, we include some results connecting the enumeration of the simple permutations to the enumeration of the permutations in $\mathcal{A}_n$ that follow from Lemma \ref{lem: infateAt1}.

\begin{theorem} \label{thm: simpleEnum}
    Let $s_n$ denotes the number of simple permutations in $\mathcal{A}_n$ and let $c_n = |\C_n(321)|$. Then for $n \geq 2$,
    \[
    a_n = \sum_{i=2}^n s_i a_{n-i+1},
    \]
    which implies that for $n \geq 2$,
    \[
    c_{n} = \sum_{i=2}^{n-1} s_i c_{n-i+1}.
    \]
\end{theorem}

\begin{proof}
    This is an immediate consequence of Lemmas \ref{lem: inflation} and \ref{lem: infateAt1}. Since every inflation yields a permutation in $\mathcal{A}_n$ by Lemma \ref{lem: inflation}, and every permutation in $\mathcal{A}_n$ is the inflation of a unique simple permutation, but there is only one way to inflate via Lemma \ref{lem: infateAt1}.
\end{proof}

We can recursively expand and get,

\begin{corollary} \label{cor: simpleComp}
    For $n \geq 2$,
    \[
    a_n = \sum_{x_1 + \cdots + x_k = n-1} s_{x_1+1}s_{x_2+1} \cdots s_{x_k+1}
    \]
    where we are summing over all compositions of $n-1$.
\end{corollary}

\begin{proof}
    An alternative proof is that Lemmas \ref{lem: inflation} and \ref{lem: infateAt1} show that every $\pi \in \mathcal{A}_n$ can be obtained by iteratively inflating the permutation $1$ by simple permutations $s_i$ such that $|s_{1}| + |s_2| - 1 + \cdots + |s_k| - 1 = n$. So we want $|s_1| + \cdots + |s_k| = n+k-1$. This precisely recovers the result. 
\end{proof}

The following result illustrates how one could approach building new simple permutations from smaller ones, and is useful for obtaining a lower bound on the growth rate of $c_n$,

 \begin{lemma} \label{lem: simple(n-2)}
    $s_n \geq 4 s_{n-2}$ for $n \geq 5$.
\end{lemma}

\begin{proof}
    Given a permutation $\pi \in A_{n-2}$ that is simple, we can form a new permutation $\pi'$ by adding $n$ directly to the right of $n-2$ and $n-1$ at the end of the permutation. We claim $\pi'$ is still simple and in $\mathcal{A}_n$.

    First, $\pi' \in \mathcal{A}_n$ since if $n$ plays the role of a $4$ in any of the vincular patterns of length $4$ from Corollary \ref{cor: 321-A}, $n-2$ would play this role in $\pi$. We avoid $\tubr{23}\, \tubr{1}$ since $\pi'$ ends with $n-1$ and $\pi'$ avoids $\tubr{1} \, \tubr{32}$ because $n-2 < n$ and if $n$ played the role of the $3$, $n-2$ would in $\pi$.

    Since $\pi' \in \mathcal{A}_n$, to show $\pi'$ is simple by Lemma \ref{lem: infateAt1}, it suffices to argue that $\pi'$ does not contain a subinterval containing $1$. First, since $\pi$ is simple, $\pi_{n-2} \not= n-2$ and $\pi_1 \not= n-2$, so $\pi'$ contains $\pi_1 \cdots (n-2)n \cdots \pi_{n-2} (n-1)$. If there were now a non-trivial subinterval containing $1$, it would have to be among the elements before $n-2$ or after $n$, but that implies $\pi$ contains a nontrivial subinterval.

    This implies that $s_n \geq s_{n-2}$, but by Lemma \ref{lem: reflect}, we can reflect all of these new simple permutations to obtain $2s_{n-2}$ permutations. They are distinct because the reflections begin with $n-1$.

    We further claim that for each of these permutations $\pi_1 \cdots \pi_{k-1}(n-2) n \pi_{k+1} \cdots \pi_{n-2} (n-1)$ that the permutation $\pi' = (n-2) \pi_{k-1} \cdots \pi_1 n \pi_{n-2} \pi_{n-3} \cdots \pi_{k+1} (n-1) \in \mathcal{A}_n$. We note that to get distinct permutations, we need $n-2 \geq 3$, i.e. $n \geq 5$. First, notice this permutation is distinct from all the others we have created because it begins with $n-2$.

    $\pi'$ still avoids $\tubr{1} \, \tubr{32}$ because $\pi'_1 = n-2$ by construction since we placed $n$ to the right of $n-2$. It avoids $\tubr{23}\, \tubr{1}$ because $\pi'_n = n-1$. 
    
    For any of the patterns of size $4$ in Corollary \ref{cor: 321-A}, were they to occur in $\pi'$ it would have to be between some elements to the right of $n$ and some to the left, otherwise $\pi$ would contain the reverse of that pattern, which is another pattern $\pi$ avoids. Furthermore, it must involve $n$ or $n-1$ because if one part of the vincular pattern appears to the left of $n$ and the other part appears to the right, $\pi$ contains the reversal of each of these patterns, which is again a pattern that $\pi$ avoids. For example, if $\pi_k \cdots \pi_1$ contains the $\tubr{14}$ and $\pi_{n-1} \cdots \pi_{k+1}$ contains the $\tubr{23}$ in a $\tubr{14} \, \tubr{23}$ pattern, then $\pi$ would contain a $\tubr{41} \, \tubr{32}$ pattern which is a contradiction. 
    
    With this in mind, $\pi'$ avoids $\tubr{32} \, \tubr{41}$ because if $n$ plays the role of a $4$ here we find a $\tubr{23}\, \tubr{14}$ pattern in $\pi$ where now $n-1$ plays the role of the $4$ in the reversal. Similar logic shows we cannot have a $\tubr{41} \, \tubr{32}$.

    We cannot create a $\tubr{14} \, \tubr{23}$ where $n$ plays the role of $4$, because this implies the existence of a $\tubr{1} \, \tubr{32}$ pattern in $\pi$. Similar logic shows we cannot create a $\tubr{41} \, \tubr{32}$.

    This covers all the possible patterns from Corollary \ref{cor: 321-A}, showing that $\pi' \in \mathcal{A}_n$. Suppose for sake of contradiction that $\pi'$ is not simple, by Lemma \ref{lem: infateAt1}, it contains a subinterval $[1,j]$ for $j < n$. This must be to the left or right of $n$, but since $\pi'$ and $\pi$ differ my reflecting all the terms to the left of $n$ and all the terms between $n$ and $n-1$, this means that $\pi$ contains the same subinterval. This contradicts the simplicity of $\pi$.

    We can also reverse all of these permutations to obtain a total of $4 s_{n-2}$ simple permutations in $\mathcal{A}_n$. They will all be distinct because the first $2 s_{n-2}$ ended or started with $n-1$, but did not start or end with $n-2$. The other $2 s_{n-2}$ permutations we produced all start or end with $n-1$ and end or start with $n-2$.
\end{proof}

\begin{example}
    As an example of the above, consider the simple permutations in $A_6$: \sloppy $246135, 314625, 362514, 413625, 415263, 426315, 513624, 526314, 526413, 531642$. We can construct $4$ of them from the two simple permutations in $A_4$, $2413, 3142$:
    \[
    2413 \to 24{\bf 6} 13 {\bf 5}, {\bf 5} 31 {\bf 6} 42 \qquad \qquad 3142 \to 314{\bf 6}2 {\bf 5}, {\bf 5} 2 {\bf 6} 413.
    \]
    We can then reverse the pieces between $6$ and $5$ to obtain four more permutations:
    \[
    42{\bf 6} 31 {\bf 5}, {\bf 5} 13 {\bf 6} 24 \qquad \qquad 413{\bf 6}2 {\bf 5}, {\bf 5} 2 {\bf 6}314
    \]
    We only miss
    \[
    362514, 415263.
    \]
\end{example}

\section{Growth rate results} \label{sec:bounds}
In this section we explore some of the growth rate results about $c_n$ that stem from the results in the previous sections.

Lemma \ref{lem: simple(n-2)} implies that $s_n$ has a growth rate of at least $2$ for $n \geq 7$, but from explicitly calculating the number of simple permutations, it seems much higher than that. The simple permutations also seem to exhibit a larger growth rate sooner than the $c_n$. For example $\frac{s_{23}}{s_{22}} \approx 3.44037$ while $\frac{c_{23}}{c_{22}} \approx 3.37576$. We include a table of the values of $s_n$ for $n = 2,\dots,23$ calculated with Sage \cite{Sage} and for large values of $n$ via Theorem \ref{thm: simpleEnum} from the known values of $c_n$ in \cite[A309508]{OEIS}, since $s_n = c_{n+1} - \sum_{i=2}^{n-1} s_i c_{n-i+2}$:

\begin{table}[H] \label{tab: simples}
\begin{center}
\begin{tabular}{c|c||c|c||c|c||c|c}
    $n$ & $s_n$ & $n$ & $s_n$ & $n$ & $s_n$ & $n$ & $s_n$ \\
    \hline
    2  & 2       & 8  & 56       & 14 & 35,950     & 20 & 47,592,074 \\
    3  & 0       & 9  & 94       & 15 & 112,834    & 21 & 161,888,174 \\
    4  & 2       & 10 & 406      & 16 & 378,818    & 22 & 555,182,652 \\
    5  & 0       & 11 & 1,000     & 17 & 1,240,626   & 23 & 1,910,032,910 \\
    6  & 10      & 12 & 3,656     & 18 & 4,180,576   &    &           \\
    7  & 6       & 13 & 10,478    & 19 & 14,003,702  &    &           \\
\end{tabular}
 \caption{Number of simple permutations in $\mathcal{A}_n$}
  \label{fig:simples}
\end{center}
\end{table}

An immediate consequence of Theorem \ref{thm: simpleEnum} and Lemma \ref{lem: simple(n-2)} is that
\begin{align*}
c_n &\geq 2 c_{n-1} + 2 c_{n-3} + 10 c_{n-5} + 6 c_{n-6} + 56 c_{n-7} + 94 c_{n-8}+ 406 c_{n-9}+ 1000 c_{n-10}\\
&  + 3656 c_{n-11} + 10478 c_{n-12} + 35950 c_{n-13} + 112834 c_{n-14} + 378818 c_{n-15}\\
& +1240626 c_{n-16} +4180576 c_{n-17} + 14003702 c_{n-18} + 47592074 c_{n-19}  + 161888174 c_{n-20} \\
&+555182652 c_{n-21} + 1910032910 c_{n-22} + 555182652 \sum_{k=23}^{\infty} 2^{k-21}c_{n-k} 
\end{align*}

From Table \ref{tab: simples}, we know that $s_{23} \geq 2 s_{22}$ so we can use Lemma \ref{lem: simple(n-2)} to conclude that $s_{k+1} \geq 2^{k-21} s_{22}$ for $k \geq 23$. An immediate corollary,

\begin{corollary}
    We have $\lim_{n \to \infty} \sqrt[n]{c_n} \geq 3.14101$.
\end{corollary}

We note that this is a slightly worse lower bound than the result in \cite{AGL25} of $3.17$. We include it because from this perspective there are two ways to improve the lower bound on the growth rate of $c_n$. First, one could prove better growth rate results about $s_n$, which seems achievable (e.g., Lemma \ref{lem: simple(n-2)}). For instance, we conjecture,

\begin{conjecture}
    For $n \geq 12$, $s_n \geq 9 s_{n-2}$.
\end{conjecture}

If one could prove this conjecture, without computing any other $s_n$, this would imply that $\lim_{n \to \infty} \sqrt[n]{c_n} \geq 3.232$, which is an improvement on the lower bound from \cite{AGL25}. But one could also combine improved $s_n$ growth rate results with calculating more of the $s_n$ (or $c_n$) to further improve the lower bound.

\subsection{Upper Bound on $c_n$}
Much of the work thus far concerning $c_n$ has involved finding lower bounds. One of the other advantages of Theorem \ref{thm: simpleEnum} is it makes it possible to study upper bound questions via inductive arguments. As one example, we can address the upper bound of \cite[Conjecture 5.2]{BC19}, where the authors conjecture for $n \geq 2$,
\[
c_n(321) \leq 4 c_{n-1}(321).
\]
We cannot quite prove this conjecture, but we can prove it up to an assumption that seems very reasonably true. This result also gives a way to connect local growth to global growth rates.

\begin{theorem} \label{thm: upperBound}
    If $\limsup_{n \to \infty} \sqrt[n]{s_n} \leq 3.99$ then $c_n \leq 4 c_{n-1}$ for $n \geq 2$.
\end{theorem}

\begin{proof}
    We will prove that $c_n \leq 4 c_{n-1}$ by strong induction on $n$. This is clearly the case when $n = 2,\dots,17$, we can verify this explicitly (see \cite[A309508]{OEIS}). For some $n \geq 15$, assume $c_k \leq 4 c_{k-1}$ for all $k < n$, we will now argue that $c_{n} \leq 4 c_{n-1}$. By Theorem \ref{thm: simpleEnum},
    \[
    4c_{n-1} - c_n = -s_{n-1} + \sum_{i=2}^{n-2} s_{n-i} (4c_i - c_{i+1}).
    \]
    By induction each of the $(4c_i - c_{i+1})$ terms are positive. Our claim is equivalent to this difference being positive, meaning
    \[
    \sum_{i=2}^{n-2} s_{n-i} (4c_i - c_{i+1}) \geq s_{n-1}.
    \]
    Assume for sake of contradiction that this is not the case. Then $\sum_{i=2}^{n-2} s_{n-i} (4c_i - c_{i+1}) \leq s_{n-1}.$ However, since $n \geq 17$ this implies
    \begin{align*}
    s_{n-1} &\geq 2s_{n-2} + 4s_{n-3} + 6s_{n-4} + 16 s_{n-5} + 30 s_{n-6} + 86 s_{n-7} + 200 s_{n-8} + 562 s_{n-9}\\
    &+ 1498 s_{n-10} + 4316 s_{n-11} + 12224 s_{n-12} + 36414 s_{n-13} + 108042 s_{n-14} + 329262 s_{n-15}
    \end{align*}
    where we are removing terms we know are positive by our inductive assumption. This means that $\limsup_{n \to \infty} \sqrt[n]{s_n} \geq 3.99088$. This contradicts our assumption and so we conclude $4 c_{n-1} - c_n \geq 0$. The result follows by induction.
\end{proof}

\sloppy Based on Corollary \ref{cor: simpleComp}, it does not seem unreasonable to us that one could show $\limsup_{n \to \infty} \sqrt[n]{s_n} \leq 3.99$. It is also possible that if one calculated more of the $c_n$ terms, one could adapt the proof to push the growth rate of $s_n$ above $4$, which would be a known contradiction since the growth rate of the $c_n$ terms is at most $4$. We address some of these ideas in the next section.

\section{Future Work}

There are many questions one could ask. Main among them is if one could explicitly count the number of permutations of $\S_n$ avoiding the set of vincular patterns from Theorem \ref{thm: 321-final} or Corollary \ref{cor: 321-A}, which is equivalent to counting $\C_n(321)$.

\begin{question}
    What is $|\S_{n}(\tubr{32}\, \tubr{41}, \tubr{14}\,\tubr{23}, \tubr{41}\,\tubr{32}, \tubr{23}\,\tubr{14}, \tubr{23} \, \tubr{1},\tubr{1} \, \tubr{32})|$?
\end{question}

These patterns seems nice, they are invariant up to reversals and are size at most four. It would even be interesting to enumerate subsets of these patterns, like $|\S_{n}(\tubr{32}\, \tubr{41}, \tubr{41}\,\tubr{32}, \tubr{23} \, \tubr{1},\tubr{1} \, \tubr{32})|$. This would provide upper bounds for the growth rate of $c_n$.

One could also strengthen the avoidance conditions by removing the vincular conditions on some of the patterns as we did in Corollary \ref{cor: 132-231} to obtain lower bounds on $c_n$. Based on Theorem \ref{thm: simpleEnum} and Corollary \ref{cor: simpleComp}, we could also ask for a more explicit connection between the growth rates of $s_n$ and $a_n$.

\begin{question}
    What is the $1 \leq M < 2$ so that $\lim_{n \to \infty}\sqrt[n]{a_n} = M \limsup_{n \to \infty} \sqrt[n]{s_n}$?
\end{question}

With this, one could prove explicit upper and lower bounds on the growth rate of $c_n$ by studying $s_n$. We also know that $\lim_{n \to \infty} \sqrt[n]{a_n} \leq 4$. So if one could find an $M > 1$ for which this is true, possibly from Corollary \ref{cor: simpleComp}, one could prove \cite[Conjecture 5.2]{BC19} for $321$ via Theorem \ref{thm: upperBound}. On a related note, if one could prove that $s_n \leq 3.97 s_{n-1}$ for $n \geq 21$, then
\begin{align*}
c_n &\leq 2 c_{n-1} + 2 c_{n-3} + 10 c_{n-5} + 6 c_{n-6} + 56 c_{n-7} + 94 c_{n-8}+ 406 c_{n-9}+ 1000 c_{n-10}\\
&  + 3656 c_{n-11} + 10478 c_{n-12} + 35950 c_{n-13} + 112834 c_{n-14} + 378818 c_{n-15}\\
& +1240626 c_{n-16} +4180576 c_{n-17} + 14003702 c_{n-18} + 47592074 c_{n-19}  + 161888174 c_{n-20} \\
&+555182652 c_{n-21} + 1910032910 c_{n-22} + 555182652 \sum_{k=23}^{\infty} (3.97)^{k-21}c_{n-k} 
\end{align*}

This would imply that $\lim_{n \to \infty} \sqrt[n]{c_n} \leq 3.97132$. The growth rate of $s_n$ cannot be $4$, though, because this would imply that $\lim_{n \to \infty} \sqrt[n]{c_n} > 4$, which we know is not true. This leads to the question,

\begin{question}
    Is $\lim_{n \to \infty} \sqrt[n]{c_n} < 4$?
\end{question}

This would be somewhat surprising, but does not seem completely unreasonable given Corollary \ref{cor: simpleComp} and Theorem \ref{thm: upperBound}.
Finally, it might be worthwhile to try classifying cyclic avoidance using the fundamental bijection for other patterns of size $3$.

\begin{question} \label{Q: singleVinc}
    \sloppy Given $\sigma \in \{123,231,312,132,213\}$, is there a set of vincular (or mesh) patterns $\Lambda$ so that $|\C_n(\sigma)| = |\S_{n-1}(\Lambda)|$?
\end{question}

One could also ask the same question about $\C_n(132,213)$. Even a partial answer to Question \ref{Q: singleVinc} could allow us to compare the sizes of $\C_n(\sigma)$ and $\C_n(\tau)$ for $\sigma,\tau \in \S_3$ by comparing the corresponding avoidance classes. This is one possible approach to \cite[Conjecture 5.1]{BC19}, where the authors conjecture $c_n(123) \geq c_n(132) = c_n(213) \geq c_n(321) \geq c_n(312) = c_n(231)$.




\subsection*{Acknowledgments}

We thank Kassie Archer and Christina Graves for helpful conversations throughout the writing of this paper.

\subsection*{Disclaimer}
\emph{The views expressed in this paper are those of the authors and do not reflect the official policy or position of the U.S. Naval Academy, Department of the Navy, the Department of Defense, or the U.S. Government.}

\bibliographystyle{amsplain}

\end{document}